\documentclass[12pt]{amsart}
\usepackage{latexsym,amssymb,amsmath}
\usepackage{latexsym,amssymb,amsmath,tikz,verbatim,cancel}

\numberwithin{equation}{section}
\newtheorem{theorem}{Theorem}[section]
\newtheorem{lemma}[theorem]{Lemma}

\newtheorem{corollary}[theorem]{Corollary}

\theoremstyle{definition}
\newtheorem{definition}[theorem]{Definition} 

\newtheorem{remark}[theorem]{Remark}
\newtheorem{example}[theorem]{Example}

\begin{document}

\title{On the Waldschmidt Constant of
Square-free Principal Borel Ideals}
\thanks{Last updated: \today}

\author[E. Camps Moreno]{Eduardo Camps Moreno}
\address{Department of Mathematics, Escuela Superior de F\'isica y Matem\'aticas, CDMX, M\'exico}
\email{camps@esfm.ipn.mx}

\author[C. Kohne]{Craig Kohne}
\address{Department of Mathematics and Statistics\\
McMaster University, Hamilton, ON, L8S 4L8, Canada}
\email{kohnec@math.mcmaster.ca}

\author[E. Sarmiento]{Eliseo Sarmiento}
\address{Department of Mathematics, Escuela Superior de F\'isica y Matem\'aticas, CDMX, M\'exico}
\email{esarmiento@ipn.mx}

 \author[A. Van Tuyl]{Adam Van Tuyl}
\address{Department of Mathematics and Statistics\\
McMaster University, Hamilton, ON, L8S 4L8, Canada}
\email{vantuyl@math.mcmaster.ca}

\begin{abstract}
Fix a square-free monomial $m \in S = \mathbb{K}[x_1,\ldots,x_n]$.  The square-free principal 
Borel ideal generated by $m$, denoted ${\rm sfBorel}(m)$, is
the ideal generated by all the square-free monomials
that can be obtained via Borel moves from
the monomial $m$.  We give upper and lower bounds for the Waldschmidt
constant of ${\rm sfBorel}(m)$ in terms of the support
of $m$, and in some cases, exact values.  
For any rational $\frac{a}{b} \geq 1$,
we show
that there exists a square-free principal 
Borel ideal with Waldschmidt
constant equal to $\frac{a}{b}$. 
\end{abstract}

\maketitle


\section{Introduction}

Introduced in the 1970's by Waldschmidt \cite{W} to study
points in $\mathbb{C}^n$, the Waldschmidt constant pays 
a pivotal role in the study of the asymptotic properties
of homogeneous ideals.   More formally, given a homogeneous
ideal $I \subseteq S=\mathbb{K}[x_1,\ldots,x_n]$, the 
{\it Waldschmidt constant} of $I$, denoted 
$\widehat{\alpha}(I)$, is the limit
$\lim_{s\rightarrow \infty} \frac{\alpha(I^{(s)})}{s}$.  Here,
$\alpha(J)$ denotes the smallest degree of a generator of the
ideal $J$, and $I^{(s)}$ denotes the $s$-th symbolic power 
of $I$.  Current interest in this invariant 
was partially inspired by the work
of Bocci and Harbourne \cite{BH}.
Specifically, the Waldschmidt constant can be used
to gain insight into the ideal containment
problem, that is, comparing the regular powers of an ideal
to those of its symbolic powers.  The papers \cite{B+,BGHN,WalSQF,CGS} form a small sample of recent work on the Waldschmidt constant;  an introduction
to this topic can also be found in \cite{CHHVT}.

The goal of this paper is to investigate the Waldschmidt
constant of square-free principal Borel ideals.  
Given a monomial $m$, if $x_i|m$ and $j < i$, then
we call $x_j\cdot \frac{m}{x_i}$ a {\it Borel move} of $m$.
A monomial ideal is a {\it Borel ideal} (or a strongly stable ideal) 
if for every $m \in I$, all of the Borel moves of 
$m$ are also in $I$.   A monomial ideal
$I$ is a {\it principal Borel ideal} if there is a single
monomial $m$ such that every generator of $I$ is obtained
via a Borel move of $m$.  A {\it square-free principal  Borel
ideal} generated by $m$, denoted ${\rm sfBorel}(m)$, is the square-free monomial ideal generated by all the square-free monomials
that can be obtained via Borel moves on $m$.  The study of 
(principal)
Borel ideals has a rich history; we point the reader to 
\cite{fms,h} and the references therein for more on this topic.

Because ${\rm sfBorel}(m)$ is a square-free monomial ideal, the results
of Bocci, {\it et al.} \cite{WalSQF} are applicable.  
For any square-free monomial ideal $I$, it is was shown in \cite{WalSQF} how
to use the primary decomposition
of $I$ to create a  linear optimization problem whose
optimal solution is the Waldschmidt constant.  Solving
this optimization problem, however, may prove to be quite difficult. 

As an example,
suppose we wish to compute $\widehat{\alpha}({\rm sfBorel}(m))$ for
the monomial
\begin{equation}\label{piexample}
m = x_{33215}x_{33216}\cdots x_{104348} \in \mathbb{K}[x_1,\ldots,x_{104348}].
\end{equation}
To naively apply  \cite{WalSQF}
to find the Waldschmidt constant 
for this ideal would involve solving a linear optimization problem
in $104348$ variables and $\binom{104348}{33215} \approx 5.1 \times 10^{28347}$ inequalities\footnote[1]{As a point of comparison, it is estimated that there are $10^{82}$ atoms in the universe!}.  Given the size of this problem, 
it is natural to ask what information, if any, one can
obtain about the Waldschmidt constant for this
family of ideals.

The main results of this paper place bounds on
the Waldschsmidt constant for any square-free principal Borel ideal
${\rm sfBorel}(m)$.  Our bounds are expressed in terms of the 
support of
$m = x_{i_1}\cdots x_{i_s}$, that is, $\{i_1,\ldots,i_s\}$.  Our proofs rely on Francisco, Mermin, and Schweig's \cite{fms} description of
the associated primes of ${\rm sfBorel}(m)$.  These associated
primes then allow us
to describe many of the inequalities in the linear
optimization problem given in \cite{WalSQF}, which are then used
to bound the Waldschmidt constant.   In some
cases, we are able to give exact values for
$\widehat{\alpha}({\rm sfBorel}(m))$.

The work in this paper generalizes some of our 
work in \cite{cksvt}.  Previously we showed that
the Waldschmidt constant of a principal Borel ideal generated
by $m$ is $\deg m$.  Unlike the case of principal
Borel ideals, the Waldschmidt constant of a square-free
principal Borel ideal need
not be an integer; in fact:

\begin{theorem}[Corollary \ref{rationalcor}]
Let $\frac{a}{b} \geq 1$ be a rational number.  Then
there exists a square-free principal Borel ideal $I$
such that $\widehat{\alpha}(I) = \frac{a}{b}$. 
\end{theorem}

\noindent
The above result is effective in the sense that  we can explicitly
construct the required monomial  $m$ so 
$\widehat{\alpha}({\rm sfBorel}(m)) = \frac{a}{b}$.   In fact,
the monomial $m$ of \eqref{piexample}  was chosen so that
$$\widehat{\alpha}({\rm sfBorel}(m)) = \frac{104348}{33215}.$$
The astute reader might recognize that this number is a convergent of the continued fraction expansion of $\pi$.  In fact, $\frac{104348}{33215} =3.14159265392$ agrees with $\pi$ up to $9$ digits.

We outline our paper.  Section 2 provides
the  background on square-free principal Borel
ideals and the Waldschmidt constant. In Section 3, upper bounds on the Waldschmidt constant are obtained.
In Section 4, under certain hypotheses, exact values 
of $\widehat{\alpha}({\rm sfBorel}(m))$ are obtained,
as well as a recursive approach to
finding a lower bound
(Theorem \ref{lowerboundresult}).

\noindent
{\bf Acknowledgments.} 
Camps is supported by Conacyt. Sarmiento's research is supported by SNI-Conacyt. Camps and Sarmiento are supported by PIFI IPN 20201016.
 Van Tuyl’s research is supported by NSERC Discovery Grant 2019-05412. 
 

\section{Background}

Throughout this paper
$S = \mathbb{K}[x_1,\ldots,x_n]$, where $\mathbb{K}$ is 
field of characteristic zero.  
In this section we recall the relevant background on square-free monomial ideals,
square-free Borel ideals, and the Waldschmdit constant.   

\subsection{Square-free Borel ideals}  For unexplained 
terminology about (square-free) monomial ideals, we refer the
reader to \cite{HH}.   We define
the main objects of study in this paper.

\begin{definition}
Let $X = \{m_1,\ldots,m_t\}$ be a set
of square-free monomials in $S$.  
The {\it square-free Borel ideal generated by $X$},
denoted ${\rm sfBorel}(X)$, is the square-free
monomial ideal generated by the square-free monomials
that can be obtained via Borel moves from any
monomial $m \in X$.  If $X = \{m\}$, then
we abuse notation and write ${\rm sfBorel}(m)$
for ${\rm sfBorel}(\{m\})$;  
furthermore, we call ${\rm sfBorel}(m)$ a
{\it square-free principal Borel ideal}.
\end{definition}

The {\it support} of a square-free monomial 
$m = x_{i_1}\cdots x_{i_s}$ is the set ${\rm supp}(m) = \{i_1,\ldots,i_s\}$.  For 
our future arguments, we need two tuples that can be
constructed from ${\rm supp}(m)$.

\begin{definition}\rm
    Let $m=x_{i_1}\cdots x_{i_s}$ be a square-free monomial. 
    Let 
    $$T(m)= (t_0,t_1,\ldots,t_k)$$ 
    where $t_0=s$ and $t_i=\max\{j<t_{i-1}\ |\ i_j<i_{j+1}-1\}$. 
    Furthermore, let 
    $$IT(m) = (i_{t_0},i_{t_1},\ldots,i_{t_k}).$$
\end{definition}

 \begin{remark}\label{redundantgens2}
 The following observations will hopefully help the reader with our notation.
  The $t_i$'s are 
    recording where the indices of
    $x_{i_1}x_{i_2}\cdots x_{i_s}$ are ``jumping'' by more than one.  For example, if
    $m =  x_2x_3x_5x_6x_8x_{10}$, then
    $$T(m) = (t_0,t_1,t_2,t_3) = (6,5,4,2)$$
    records the positions where the indices
    increase by more than one.  Note that we
    are recording this information from right-to-left.  Equivalently,
    we can define the $t_i$'s as follows.
    Consider the tuple $(i_1-1,i_2-2,\ldots,i_s-s)$.  The $t_i$'s
    are then the locations where 
    $i_j-j < i_{j+1}-(j+1)$, again reading
    right-to-left.  In our example,
    $(2-1,3-2,5-3,6-4,8-5,10-6) = 
    (1,1,2,2,3,4)$, so $t_0 =6$,
    and 
    $t_1=5$, $t_2=4$ and $t_3=2$ since
    these are the indices where $i_j-j < i_{j+1}-(j+1)$.
    Continuing with this example, the tuple
    $$IT(m) = (i_6,i_5,i_4,i_2) = (10,8,6,3)$$
    records the indices of the variables where the ``jump" occurs.
 \end{remark}
 
 The following lemma records some facts that follow
 immediately from the definitions.
 
 \begin{lemma}\label{lem.tm}
 Let $m$ be a square-free monomial with $T(m) = (t_0,t_1,\ldots,
 t_k)$ and $IT(m) = (i_{t_0},i_{t_1},\ldots,i_{t_k})$. Then
 \begin{enumerate}
     \item $s= t_0 > t_1 > \cdots > t_k \geq1$,
     \item $i_{t_0} > i_{t_1} > \cdots > i_{t_k}$, and
     \item $i_{t_j}-t_j > i_{t_{j+1}}-t_{j+1} $ for 
     $j=0,\ldots,k-1$.
 \end{enumerate}
 \end{lemma}
 \begin{proof}
 (1) and (2) are immediate. For (3) we have
 \begin{equation}\label{jumps}
 i_1-1 \leq i_2 -2 \leq \cdots \leq i_j-j \leq \cdots
 \leq i_s-s
 \end{equation}
 for all $j=1,\ldots,s$.  So
 $$i_{t_{j+1}} - t_{j+1} < i_{t_{j+1}+1} - (t_{j+1}+1) 
 \leq i_{t_j}-t_j
 $$
 since the $t_i$'s are precisely the locations where
 the inequalities in \eqref{jumps} are strict.
 \end{proof}

Given an ideal $I$, we let ${\rm ass}(I)$ denote the {\it set of
associated primes} of the ideal $I$.  When $I$ is a square-free
monomial ideal, it is known that all the associated primes
are prime monomial ideals.
The next result is critical
for our arguments, since the associated primes
of ${\rm sfBorel}(I)$ will be related to
an optimization problem to compute the Waldschmidt constant.
Note that the original statement 
of Theorem \ref{theo.dual1} involved the language of
Alexander duals.  We have given an equivalent expression
for this statement.

\begin{theorem} \cite[Theorem 3.17]{fms} \label{theo.dual1}
Let $m=x_{i_1}\cdots x_{i_s}$ be a square-free monomial 
with $T(m) = (t_0,t_1,\ldots,
 t_k)$ and $IT(m) = (i_{t_0},i_{t_1},\ldots,i_{t_k})$, 
 and suppose $I = {\rm sfBorel}(m)$.
    Then 
    \[\langle x_{j_1},\ldots,x_{j_l} \rangle \in {\rm ass}(I) \]
    if and only if $x_{j_1}x_{j_2}\cdots x_{j_l}$ is a
    minimal generator of the square-free Borel ideal
    $${\rm sfBorel}(\{x_{t_k}x_{t_k+1} \cdots x_{i_{t_k}},~
    x_{t_{k-1}}x_{t_{k-1}+1} \cdots x_{i_{t_{k-1}}}, ~\ldots,~
    x_{t_0}x_{t_0+1}\cdots x_{i_{t_0}}\}).$$
\end{theorem}

\noindent
The monomials
$x_{t_j}\cdots x_{i_{t_j}}$ for $j=0,\ldots,k$ are minimal in
the sense that if we remove any of them, we  change the 
generators of the resulting square-free Borel ideal.

\begin{example}
   In Remark \ref{redundantgens2} it was shown that the monomial
   $m=x_2x_3x_5x_6x_8x_{10}$ has $T(m) = (6,5,4,2)$ and
   $IT(m) = (10,8,6,3)$.  So the associated primes of 
   ${\rm sfBorel}(m)$ are in one-to-one correspondence with
   the minimal generators of 
   $${\rm sfBorel}(\{x_2x_3, x_4x_5x_6, x_5x_6x_7x_8, x_6x_7x_8x_9x_{10}\}).$$
   \end{example}

\subsection{The Waldschmidt constant}  We recall the definition
of the Waldschmidt constant and a procedure to compute this invariant
for square-free monomial ideals.

Given a square-free monomial ideal $I \subseteq S$,  let 
$I = P_1 \cap \cdots \cap P_t$ denote its minimal primary decomposition.
The {\it $s$-th symbolic power} of the square-free monomial ideal 
$I$, denoted $I^{(s)}$,  is the ideal
$$I^{(s)} = P_1^s \cap P_2^s \cap \cdots \cap P_t^s.$$
Note that there is a more general definition of a symbolic power
of an ideal (see \cite{CHHVT,DDGHN}); our definition is 
equivalent when restricted to square-free monomial ideals.  

For any homogeneous ideal $J \subseteq S$, let $\alpha(J)$
denote the smallest degree of a generator of $J$.  The
{\it Waldschmidt constant} of a square-free monomial
ideal $I$, denoted $\widehat{\alpha}(I)$, is then
$$\widehat{\alpha}(I) = \lim_{s \rightarrow \infty} \frac{\alpha(I^{(s)})}{s}.$$
Our key tool will be the following result which relates
the Waldschmidt constant of a square-free monomial ideal
to a linear program.

\begin{theorem}\cite[Theorem 3.2]{WalSQF}\label{linprogthm1}
Let $I \subseteq S = \mathbb{K}[x_1,\ldots,x_n]$ be a square-free monomial ideal with minimal primary decomposition $I = P_{1} \cap P_{2} \cap \dots \cap P_{t}$. Define the $t \times n$ matrix $A$ where 
\[ A_{i,j} = \begin{cases} 1 &\mbox{if } x_{j} \in P_{i}\\
0 & \mbox{if } x_{j} \notin P_{i}. \end{cases}\]
Then $\widehat\alpha(I)$ is the optimum value of the linear program
\[\min\{\mathbf{1}^{T}\mathbf{y}\ |\ A\mathbf{y}\geq\mathbf{1}, \mathbf{y}\geq\mathbf{0}\}.\]
\end{theorem}
\noindent
The matrix $A$ in the above theorem
will be called the \emph{matrix of associated primes of $I$}. 

\begin{example}
    Let $m$ be the monomial of \eqref{piexample} from the introduction,
    and so $T(m) = (71134)$ and $IT(m) = (104348)$.  The associated primes of ${\rm sfBorel}(m)$ are in one-to-one
    correspondence with the generators of
    ${\rm sfBorel}(x_{71134}x_{71135}\cdots x_{104348})$.  
    But this is the ideal generated by
    all the square-free monomials of degree
    $33215$, of which there are $\binom{104348}{33215}$.  So
    the matrix of associated primes of ${\rm sfBorel}(m)$
    will be a $\binom{104348}{33215} \times 104348$ matrix.
\end{example}


\section{Upper bounds}

In this section we give an upper bound on the Waldschmidt
constant of a square-free principal Borel ideal.  Our strategy
is to show that there is enough structure
in the optimization problem of Theorem \ref{linprogthm1}
that we can bound the Waldschmidt constant.

We begin with a lemma which allows us to reduce to a smaller
polynomial ring.

\begin{lemma}\label{reducevar}
Let $m=x_{i_1}\cdots x_{i_s}$ be a square-free monomial in
$S = \mathbb{K}[x_1,\ldots,x_n]$ with $I = {\rm sfBorel}(m) \subseteq S$.  Consider
the same monomial $m$, but in the ring $R =\mathbb{K}[x_1,\ldots,x_{i_s}]$, and
let $J = {\rm sfBorel}(m) \subseteq R$.  Then
$\widehat{\alpha}(I) = \widehat{\alpha}(J).$
\end{lemma}

\begin{proof}
By Theorem \ref{theo.dual1}, the associated primes of $I$ and $J$ are 
the same (although viewed in different rings).  So the matrix of
associated primes of $J$ in Theorem \ref{linprogthm1} is the same
as the matrix of the associated primes of $I$, except that the
columns in matrix of associated primes of $I$ indexed by the variables $x_{i_s+1},\ldots,x_n$ all contain
zeroes.  The result now follows from Theorem \ref{linprogthm1}.
\end{proof}

Before proceeding, we introduce additional notation.
Given a square-free monomial $m$ with
$T(m) = (t_0,\ldots,t_k)$ and $IT(m) = (i_{t_0},
\ldots,i_{t_k})$,
we have the  inequalities
$$t_k < t_{k-1} < \cdots < t_0 = s~~\mbox{and}~~
i_{t_k} < i_{t_{k-1}} < \cdots < i_{t_0}.$$  
by Lemma \ref{lem.tm}.
Let $\ell$ be smallest integer such that 
$$i_{t_{\ell+1}} < t_0 \leq i_{t_\ell}.$$
In particular, $\ell$ identifies where in the sequence of $i_{t_j}$'s
we would place $t_0=s$.  

Let $A$ be the matrix of associated primes of 
$I = {\rm sfBorel}(m)$.  We will let 
$A_P$ denote the row associated to the associated prime $P$.  The row $A_P$ corresponds to
a minimal generator $m$ of the ideal in Theorem \ref{theo.dual1}.  

As the next lemma shows, we can bound
the optimal solution of Theorem \ref{linprogthm1}
by considering only a submatrix of
the matrix of associated primes.

\begin{lemma}\label{submatrix}
Let $m=x_{i_1}\cdots x_{i_s}$ be a square-free monomial 
with $T(m) = (t_0,t_1,\ldots,
 t_k)$, $IT(m) = (i_{t_0},i_{t_1},\ldots,i_{t_k})$, and $\ell$ as 
 defined above. Let $I = {\rm sfBorel}(m)$, and let $A$ denote
 its matrix of associated primes.
Let $B$ be the submatrix of $A$ where
the $j$-th row of $B$ corresponds to
the associated prime 
$\langle x_{t_j},\ldots, x_{i_{t_j}}\rangle$
for $j=0,\ldots,k$.
Suppose $\mathbf{x} \in \mathbb{R}^n$ is such that
\begin{enumerate}
\item $B\mathbf{x}\geq\mathbf{1}$,
\item $\mathbf{x}_j\geq\mathbf{x}_{j+1}$ for $1\leq j\leq i_{t_\ell}$, and
\item $\mathbf{x}_{i_{t_\ell}}\geq\mathbf{x}_j$, for $i_{t_\ell} \leq j \leq n$.
\end{enumerate}
Then $A\mathbf{x}\geq\mathbf{1}$.
\end{lemma}

\begin{proof}
By Lemma \ref{reducevar}, we can assume
$n = i_s$.   Consider any row $A_P$ of
$A$.  By Theorem \ref{theo.dual1}, $P$ corresponds
to a monomial $m$ that is a Borel move
of exactly one of $\{
x_{t_k}\cdots x_{i_{t_k}},\ldots,
x_{t_0}\cdots x_{i_{t_0}}\}$.  Say
$m$ is a Borel move of  $x_{t_j}\cdots x_{i_{t_j}}$.

The $j$-th row of $B$ (which corresponds
to $x_{t_j}\cdots x_{i_{t_j}}$) is given by
$$B_j = (\underbrace{0,\ldots,0}_{t_j-1},\underbrace{1,\ldots,1}_{i_{t_j}-t_j+1},0,\ldots,0).$$
The rows  $A_P$ and $B_j$ have the same
number of $1$'s.   Additionally, $A_P$ is formed from
$B_j$ by swapping some of the $1$'s with
some of the $0$'s among the first $t_j-1$ spots.

Since $B_j{\bf x} \geq 1$, we have
${\bf x}_{t_j}+ \cdots + {\bf x}_{i_{t_j}} \geq 1.$
Note that $A_P{\bf x}$ is formed from 
$B_j{\bf x}$ by subtracting some ${\bf x}_{p}$'s with
$p \in \{t_j,\ldots,{i_{t_j}}\}$ and adding
in some ${\bf x}_q$'s with $q \in \{1,\ldots,{t_j-1}\}$.

If $p \in \{t_j,\ldots,i_{t_\ell}\}$, then
the hypotheses imply that ${\bf x}_q \geq {\bf x}_p$
for all $q \in \{1,\ldots,t_j-1\}$.
If $p \in \{i_{t_\ell},\ldots,i_{t_j}\}$,
then $t_j < t_0 \leq i_{t_{\ell}} \leq p
\leq i_{t_j}$.  But then 
${\bf x}_q \geq {\bf x}_{i_{t_\ell}} \geq {\bf x}_p$
for all $q \in \{1,\ldots,t_j-1\}$.  
But this means 
$$A_P{\bf x} \geq B_j{\bf x} \geq 1$$
because every time we subtract an ${\bf x}_p$
with $p \in \{t_j,\ldots,i_{t_j}\}$ we are
replacing it with an ${\bf x}_q$ with
$q \in \{1,\ldots,t_{j}-1\}$ which is larger.

The result now follows since $A_P{\bf x} \geq 1$
for all rows of $A$.
 \end{proof}

We can now bound
the Waldschmit constant of a square-free principal
Borel ideal in terms of $T(m),IT(m)$, and $\ell$.

\begin{theorem}\label{theo.aw}
Let $m=x_{i_1}\cdots x_{i_s}$ be a square-free monomial 
with $T(m) = (t_0,t_1,\ldots,
 t_k)$ and $IT(m) = (i_{t_0},i_{t_1},\ldots,i_{t_k})$, 
 and suppose $I = {\rm sfBorel}(m)$.  If 
 $\ell$ is the smallest integer such that $i_{t_\ell+1}
 < t_0 \leq i_{t_\ell}$, then
\begin{eqnarray*}
\widehat{\alpha}(I) & \leq & (t_0 -t_{\ell})\left(\frac{1}{i_{t_\ell}-t_\ell+1}\right)
+   (i_{t_\ell}-i_{t_{\ell+1}})\left(\frac{1}{i_{t_\ell}-t_{\ell}+1} \right) 
\\
&&
 + \cdots +(i_{t_{k-1}}-i_{t_k})\left(\frac{1}{i_{t_{k-1}}-t_{k-1}+1}
    \right) + i_{t_k}\left(\frac{1}{i_{t_k}-t_k+1}
    \right).
\end{eqnarray*}
\end{theorem}

\begin{proof}  By Lemma \ref{reducevar}, we can
assume that the number of variables is $i_s =n$.
Set $a = i_{t_\ell}-t_\ell+1$, and consider the vector 
${\bf y} \in \mathbb{R}^{i_s= i_{t_0}}$ where
\begin{eqnarray*}
\mathbf{y}^T & =& \underbrace{\left(\frac{1}{i_{t_k}-t_k+1},\ldots,\frac{1}{i_{t_k}-t_k+1}\right.}
 _{i_{t_k}},
 \underbrace{\frac{1}{i_{t_{k-1}}-t_{k-1}+1},\ldots,\frac{1}{i_{t_{k-1}}-t_{k-1}+1}}_{i_{t_{k-1}}-i_{t_k}},\ldots, \\
&& \ldots
\underbrace{\frac{1}{a},\ldots,\frac{1}{a}}
 _{i_{t_\ell}-i_{t_\ell+1}},
 \underbrace{\frac{(t_{\ell-1}-t_\ell)}{(i_{t_{\ell-1}}-i_{t_\ell})a},\ldots,\frac{(t_{\ell-1}-t_\ell)}{(i_{t_{\ell}-1}-i_{t_\ell})a}}_{i_{t_{\ell-1}}-i_{t_\ell}},
 \underbrace{\frac{(t_{\ell-2}-t_{\ell-1})}{(i_{t_{\ell-2}}-i_{t_{\ell-1}})a},\ldots,\frac{(t_{\ell-2}-t_{\ell-1})}{(i_{t_{\ell-2}}-i_{t_{\ell-1}})a}}_{i_{t_{\ell-2}}-i_{t_{\ell-1}}},\\
&&\ldots,
 \underbrace{\left. \frac{(t_{0}-t_{1})}{(i_{t_{0}}-i_{t_1})a},\ldots,\frac{(t_{0}-t_{1})}{(i_{t_0}-i_{t_{1}})a}\right)}_{i_{t_{0}}-i_{t_{1}}}
.
\end{eqnarray*}

Let $A$ be the matrix of associated primes of $I$. 
We will use Lemma \ref{submatrix} to verify that
$A{\bf y} \geq {\bf 1}$.  Theorem \ref{linprogthm1} then
gives the required result after we sum all the entries of 
${\bf y}$.

It follows from Lemma \ref{lem.tm} (3) that 
\[\frac{1}{i_{t_j}-t_j+1} \geq 
\frac{1}{i_{t_{j-1}}-t_{j-1}+1}~~\mbox{for all
$j=1,\ldots,k$}.\]
These inequalities imply that the first $i_{t_\ell}$ entries  of ${\bf y}$ form a non-increasing sequence.  Thus, condition (2)
of Lemma \ref{submatrix} holds for ${\bf y}$.
In addition, it follows by Lemma \ref{lem.tm}  that 
$$\frac{t_{j-1}-t_j}{i_{t_{j-1}}-i_{t_j}} < 1 ~~\mbox{for $j=1,\ldots,\ell$.}$$
Hence $\frac{1}{a} \geq {\bf y}_r$ for all $r=i_{t_\ell},
\ldots,i_{t_0}$, and thus condition (3) of Lemma
\ref{submatrix} also holds for ${\bf y}$.

Let $B$ be the submatrix of $A$ where the $j$-th row
of $B$ corresponds to the associated prime of $I$ that 
is associated to $x_{t_j}\cdots x_{i_{t_j}}$.  
That is,  written 
as a row vector:
 $$B_j =(\underbrace{0,\ldots,0}_{t_j-1},\underbrace{1,\ldots,1}_{i_{t_j}-t_j+1},0,\ldots,0).$$
We now
show that $B$ and ${\bf y}$ satisfy condition (1) of
Lemma \ref{submatrix}, thus completing the proof.

Consider the $j$-th row of $B$, denoted $B_j$.  If $j \geq \ell$,
we then have
$$B_j\mathbf{y}\geq
B_j\begin{pmatrix}
{\bf 0} \\
\frac{1}{i_{t_j}-t_j+1}\\
\vdots\\
\frac{1}{i_{t_j}-t_j+1}\\
\mathbf{y}_{i_{t_j}+1}^{i_{t_0}}
\end{pmatrix} \geq 1$$
where $\mathbf{y}_{i_{t_j}+1}^{i_{t_0}}$ represents the last $i_{t_0}-i_{t_j}$ entries of $\mathbf{y}$,
the fraction $\frac{1}{i_{t_j}-t_j+1}$ appears
$i_{t_j}-t_j+1$ times, and the ${\bf 0}$ 
is the vector with $t_j-1$ zeroes.
Since every entry
of this new vector is less than or equal to the corresponding entry in ${\bf y}$, the first inequality holds.

Now suppose that $j < \ell$. 
Consequently, note that $t_j < t_0 \leq i_{t_\ell} <i_{t_j}$.  So we then have
    \begin{align*}
    B_j\mathbf{y}&=\sum_{r=1}^{i_s} (B_j)_{r}\mathbf{y}_r =\sum_{r=t_j}^{i_{t_\ell}}\mathbf{y}_r+\sum_{r=i_{t_\ell}+1}^{i_{t_j}}\mathbf{y}_r\\
    & \geq \frac{i_{t_\ell}-t_j+1}{i_{t_\ell}-t_\ell+1}+
    \frac{t_{\ell-1}-t_{\ell}}{i_{t_\ell}-t_\ell+1} + \cdots +
    \frac{t_{j}-t_{j+1}}{i_{t_\ell}-t_{\ell}+1} \\
    &=\frac{i_{t_\ell}-t_j+1}{i_{t_\ell}-t_\ell+1}+\frac{t_j-t_\ell}{i_{t_\ell}-t_\ell+1} =1.
    \end{align*}
The inequality follows from the fact that
${\bf y}_r \geq \frac{1}{a}$ for all $1 \leq r \leq i_{t_\ell}$.
Hence, $B_j{\bf y} \geq 1$ for all rows of $j$, so condition
(1) of Lemma \ref{submatrix} also holds.
\end{proof}

We derive the following corollary; in the next section
we will show that this bound is exact under additional
hypotheses.

\begin{corollary}\label{cor.upperbound}
    Let $m=x_{i_1}\cdots x_{i_s}$ be a square-free monomial 
with $T(m) = (t_0,t_1,\ldots,
 t_k)$ and $IT(m) = (i_{t_0},i_{t_1},\ldots,i_{t_k})$, 
 and suppose $I = {\rm sfBorel}(m)$.  If 
 $\ell$ is the smallest integer such that $i_{t_\ell+1}
 < t_0 \leq i_{t_\ell}$, then
 $$\widehat{\alpha}(I) \leq \frac{t_0-t_\ell+i_{t_\ell}}{i_{t_k}-t_k+1}.$$
\end{corollary}

\begin{proof}
Recall that $t_0 = s$.
Note that $$\frac{1}{i_{t_k}- t_k+1} \geq 
\frac{1}{i_{t_j}-t_j+1} ~~\mbox{for $\ell \leq j \leq k$}.$$
The result now follows from Theorem \ref{theo.aw} and this inequality.
\end{proof}

\begin{example} Our bound in Theorem \ref{theo.aw} is sharp.
For example, if $m=x_2x_3x_5x_6x_8x_{10}$, then we have $T(m)=(6,5,4,2)$ and $IT(m)=(10,8,6,3)$.  For
this monomial, $\ell =2$ since $t_2=4$ and 
$t_0 = 6 \leq i_{4} = i_{t_\ell}=6$. Then
    $$\hat\alpha(\mathrm{sfBorel}(m))\leq\frac{2}{3}+\frac{3}{3}+\frac{3}{2}=\frac{19}{6},$$
    and this is the actual Waldschmidt constant.
\end{example}


\section{Some exact values and  lower bounds}

In this section we compute the exact value of 
the Waldschmidt constant 
of  principal square-free Borel ideals under some
additional hypotheses.  We then present a theorem
that can be used to find lower bounds recursively.

We begin with the following exact formula.

 \begin{theorem}\label{aw}
  Let $m=x_{i_1}\cdots x_{i_s}$ be a square-free monomial 
with $T(m) = (t_0,t_1,\ldots,
 t_k)$ and $IT(m) = (i_{t_0},i_{t_1},\ldots,i_{t_k})$, 
 and suppose $I = {\rm sfBorel}(m)$.  If 
$t_0 \leq i_{t_k}$, then
 $$\widehat{\alpha}(I) = 1 + \frac{s-1}{i_{t_k}-t_k+1}.$$
\end{theorem}

\begin{proof}
The hypotheses imply that $\ell =k$,
with $\ell$ as in Corollary \ref{cor.upperbound}.
Consequently, Corollary \ref{cor.upperbound} then shows that
$$\widehat{\alpha}(I)
\leq 
\frac{t_0-t_{k}+i_{t_k}}{i_{t_k}-t_k+1}
= \frac{s-1 + (i_{t_k}-t_k+1)}{i_{t_k}-t_k+1}$$
where we use the fact that $t_0 = s$.

For $j=0,\ldots,k$, let $P_j$
denote the associated
prime of $I$ that is associated
with the monomial $x_{t_j}\cdots x_{i_{t_j}}$
using the correspondence of Theorem \ref{theo.dual1}.  For any
other associated prime $P \in {\rm ass}(I)$,
we will write $P \sim P_j$ if the prime $P$
is associated to a monomial $m$ that can be
obtained from $x_{t_j}\cdots x_{i_{t_j}}$ via
a Borel move.  Note that each associated prime
satisfies $P \sim P_j$ for exactly one
$j \in \{0,\ldots,k\}$.

Let $A$ be the matrix of associated primes
of $I$.  There are $|{\rm ass}(I)|$ rows,
and we write $A_P$ for the row indexed
by the associated 
prime $P$.
Let $\mathbf{x}\in\mathbb{R}^{|\mathrm{ass}(I)|}$.  We will write ${\bf x}_P$ to denote
the corresponding coordinate in ${\bf x}.$
That is, if $P$ indexes the $i$-th row
of $A$, then ${\bf x}_P$ denotes the $i$-th
coordinate of ${\bf x}$.

With the above notation, we 
now define the vector ${\bf y} \in \mathbb{R}^{|{\rm ass}(I)|}$ as follows:
$$\mathbf{y}_P=\begin{cases}
\frac{1}{\binom{i_{t_k}-1}{t_k-1}}\frac{s-1}{i_{t_k}}& P\sim P_k\\
\frac{1}{\binom{i_{t_k}}{s-1}}& P\sim P_0~~
\mbox{and $\langle x_{i_{t_k}+1},\ldots,x_{i_s}\rangle \subseteq P$}
\\
0&\text{otherwise.}
\end{cases}$$
Note that the second criterion means we are only interested in those
prime ideals that arise from Borel moves of $x_{t_0}\cdots x_{i_{t_0}}$
that also contain the variables $\{x_{i_{t_k}+1},\ldots,
x_{i_s = i_{t_0}}\}$.

The $r$-th row of $A^T$ is such that $(A^T)_{r,P}=1$ if $x_r \in P$ and $0$ otherwise. Then
$$(A^T\mathbf{y})_r=\sum_{x_r \in P}\mathbf{y}_P.$$
 So, in order to compute $A^T\mathbf{y}$, we have to compute how many times $x_r$ appears in some $P$ such that $P\sim P_k$, or $P\sim P_0$ and 
 $ \langle x_{i_{t_k}+1},\ldots,x_{i_s} \rangle \subseteq P $.

Observe that there are $\binom{i_{t_k}}{t_k-1}$ associated primes $P$ such that $P\sim P_k$.   This number
is the number of Borel moves that can be made from
$x_{t_k}\cdots x_{i_{t_k}}$.  There are $\binom{i_{t_k}}{s-1} = 
\binom{i_{t_k}}{i_{t_k}+1-s}$ associated primes $P$ such that $P\sim P_0$ and $\langle x_{i_{t_k}+1},\ldots,x_{i_s} \rangle \subseteq P$.   To see why this is true, suppose
that we consider a Borel move of $x_{t_0}\cdots x_{i_{t_0}}$
that is also divisible by $x_{i_{t_k}+1}\cdots x_{i_{t_0}}$, i.e.,
the Borel move has the form ${m'}(x_{i_{t_k}+1}\cdots x_{i_{t_0}})$ where
${m'}$ is a degree $i_{t_k}-t_0+1$ monomial in $\{x_1,\ldots,x_{i_{t_k}}\}$.
Since $s =t_0 \leq i_{t_k}$,  there are $\binom{i_{t_k}}{s-1}\geq 1$ possible
$m'$.

For $i_{t_k}<r \leq i_{t_0}$ we have that $x_r$ appears in every $P$ such that $\mathbf{y}_P\neq 0$ and $P\sim P_0$. Therefore
$$(A^T\mathbf{y})_r=\binom{i_{t_k}}{s-1} \frac{1}{\binom{i_{t_k}}{s-1}} = 1.$$

Now, for $1\leq r \leq i_{t_k}$, $x_r$ appears in $\binom{i_{t_k}-1}{t_k-1}$ elements $P$ such that $P\sim P_k$, and in $\binom{i_{t_k}-1}{s-1}$ elements $P$ such that $P\sim P_0$ and $\langle x_{i_{t_k}+1},\ldots,x_{i_s} \rangle \subset P$. Therefore
\begin{align*}
(A^T\mathbf{y})_r &=\binom{i_{t_k}-1}{t_k-1}\left(
\frac{1}{\binom{i_{t_k}-1}{t_k-1}}\frac{s-1}{i_{t_k}}\right)+
\binom{i_{t_k}-1}{s-1}
\frac{1}{\binom{i_{t_k}}{s-1}}\\
&=\frac{s-1}{i_{t_k}}+\frac{i_{t_k}-s+1}{i_{t_k}} =1.
\end{align*}

This proves that $A^T\mathbf{y}=\mathbf{1}$. Finally
\begin{align*}
\mathbf{y}^T\mathbf{1}&=\binom{i_{t_k}}{t_k-1}
\left(\frac{1}{\binom{i_{t_k}-1}{t_k-1}}\frac{s-1}{i_{t_k}}\right)+
\binom{i_{t_k}}{s-1}
\left(\frac{1}{\binom{i_{t_k}}{s-1}}\right)\\
&=\frac{i_{t_k}}{i_{t_k}-t_k+1}\frac{s-1}{i_{t_k}}+1 =1+\frac{s-1}{i_{t_k}-t_k+1}.
\end{align*}

Due to the duality theorem, we can conclude the result.
\end{proof}

 We arrive at the following corollary which was highlighted in 
 the introduction.

\begin{corollary}\label{sucvar}\label{rationalcor}
Let $I = {\rm sfBorel}(x_ix_{i+1}\cdots x_{i+l})$.  Then
    $$\widehat\alpha(I)=\frac{i+l}{i}.$$
Consequently, for every rational number
$\frac{a}{b} \geq 1$, 
there exists a square-free principal Borel ideal $I$
such that $\widehat{\alpha}(I) = \frac{a}{b}$. 
\end{corollary}

\begin{proof}
    We have $T(m)=(l+1)$ and 
    $IT(m)=(i+l)$. Now apply Theorem \ref{aw}.

For the second statement,
if $\frac{a}{b} = 1$, we can take $I = {\rm sfBorel}(x_1) = \langle x_1
\rangle$, from which it follows that $\widehat{\alpha}(I) = 1$.
If $\frac{a}{b} > 1$, i.e., $a>b$, we have $\frac{a}{b} = \frac{b+(a-b)}{b}$.
Then the result follows if we take $m = x_bx_{b+1}\cdots x_a =
x_b x_{b+1} \cdots x_{b+(a-b)}$.
\end{proof}

\begin{remark}
When $I = {\rm sfBorel}(x_ix_{i+1}\cdots x_n)$, then $I$ is generated
by all the square-free monomials of degree $n-i+1$ in $S$.  The Waldschmidt
constants for these ideals were first computed in
\cite[Theorem 7.5]{WalSQF}.
\end{remark}

We  now give a lower bound for the Waldschmidt
constant of a square-free principal Borel ideal
in terms of a smaller square-free principal
Borel ideal.

\begin{theorem}\label{lowerboundresult}
Let $m=x_{i_1}\cdots x_{i_s}$ be a square-free monomial 
with $T(m) = (t_0,t_1,\ldots,
 t_k)$ and $IT(m) = (i_{t_0},i_{t_1},\ldots,i_{t_k})$, 
 and suppose $I = {\rm sfBorel}(m)$.  Suppose that
 $\ell$ is the smallest integer such that $i_{t_{\ell+1}}
 < t_0 \leq i_{t_\ell}$. Define $\nu=i_{t_{\ell+1}}+1$. Then
 \[\widehat{\alpha}(I) \geq \widehat{\alpha}({\rm sfBorel}(x_{i_1}\cdots x_{i_{t_{\ell+1}}}))
+1 + \frac{t_0-\nu}{i_{\nu}-\nu+1} .\]
\end{theorem}

\begin{proof}

By Lemma~\ref{reducevar}, we can assume we are working in
the polynomial ring  $\mathbb{K}[x_1,\ldots,x_{i_{t_0}}]$.
Consider the monomials
$$m_1 = x_{i_1}x_{i_2}\cdots x_{i_{t_{\ell+1}}} \in \mathbb{K}[x_1,\ldots,x_{i_{t_{\ell+1}}}]$$
and
$$m_2 = x_{i_{\nu}}x_{i_{\nu}+1}\cdots x_{i_{t_0}} \in \mathbb{K}[x_\nu,\ldots,x_{i_{t_0}}].$$
Observe that while $m_1m_2|m$, $m$ is not necessarily this product.

Let $I_1 = {\rm sfBorel}(m_1)$ and $I_2 = {\rm sfBorel}(m_2)$, in their respective 
rings, and furthermore, let $A(I_1)$ and $A(I_2)$ be the corresponding matrices
of associated primes.

Take $p$ to be the biggest integer such that $\nu\leq t_{p}$.  Then
$$T(m_2) = (t_0-\nu+1,t_1-\nu+1,\ldots,t_{p}-\nu+1)$$ 
and 
$$IT(m_2) = (i_{t_0}-\nu+1,i_{t_1}-\nu+1,\ldots,i_{t_{p}}-\nu+1).$$  
We have $p\leq\ell$, then $i_{t_{p}}\geq t_0$. So by Theorem~\ref{aw} we have

$$\widehat{\alpha}(I_2) = 1 + \frac{t_0-\nu}{i_{t_{p}}-t_{p}+1}.$$

Observe that $i_\nu-\nu+1=i_{t_p}-t_p+1$, since $t_{p+1}<\nu\leq t_{p}$, meaning that $x_{\nu}\cdots x_{i_\nu}$ is a Borel movement of $x_{t_p}\cdots x_{i_{t_p}}$.

We claim that
$$\begin{bmatrix} A(I_1)&\mathbf{0}\\ \mathbf{0}&A(I_2)\end{bmatrix}$$

\noindent is a submatrix of $A(I)$, the matrix of associated primes of $I$.

First notice that any row of $A(I_1)$ is comes from a Borel movement of a corresponding $x_tx_{t+1}\cdots x_{i_t}$ for some $t\in T(m_1)$. By Theorem~\ref{theo.dual1}, these are also associated primes of $I$, implying that $[A(I_1)\ \mathbf{0}]$ is a submatrix of $A(I)$.

Now, for any $\ell+1>u>p$, $x_{t_u}\cdots x_{i_{t_u}}$ corresponds to a row of $A(I)$, and any of its Borel movements contain at least one $x_j$ with $j<\nu$, otherwise $\nu\leq t_u$, contradicting the choice of $p$. This means for a row $R$ in $[\mathbf{0}\ A(I_2)]$,
there cannot exist a row $R'$ of $A(I)$ with $\mathrm{supp}\ R'\subsetneq\mathrm{supp}\ R$. Thus any associated prime of $I_2$ can be viewed as an associated prime of $I$ by Theorem~\ref{theo.dual1}. Thus $[\mathbf{0}\ A(I_2)]$ is a submatrix of $A(I)$ and by our choice of $\nu$.

Let $\mathbf{y}_1$ and $\mathbf{y}_2$ be such that
    $$A(I_1)^T\mathbf{y}_1\leq\mathbf{1},~~ \mathbf{1}^T\mathbf{y}_1=\widehat\alpha(I_1) ~~\mbox{and}~~
    A(I_2)^T\mathbf{y}_2\leq\mathbf{1},~~ \mathbf{1}^T\mathbf{y}_2=\widehat\alpha(I_2).$$
    After permuting rows, we can assume that
    $$A(I) = \begin{bmatrix}
    A(I_1)&\mathbf{0}\\
    \mathbf{0}&A(I_2)\\
    B_1 & B_2\end{bmatrix}$$
where $B_1,B_2$ are some appropriately sized matrices.  Set 
${\bf z} = ({\bf y}_1, {\bf y}_2,{\bf 0})$, where ${\bf 0}$ is a 
vector of zeroes, where the number of zeroes in this vector are the same
as the number of rows as $B$.  Then $A(I)^T\mathbf{z}\leq\mathbf{1}$.
Thus, by the dual version of Theorem \ref{linprogthm1}, we have
$$ \widehat{\alpha}(I) \geq \widehat{\alpha}(I_1) + \widehat{\alpha}(I_2)
= \widehat{\alpha}(I_1) +1+\frac{t_0-\nu}{i_\nu-\nu+1}.$$
\end{proof}

Theorem \ref{lowerboundresult} reduces the problem of finding
a lower bound on the principal square-free Borel ideal
$m = x_{i_1}\cdots x_{i_s}$ to finding a lower bound on the
principal square-free Borel ideal of $x_{i_1}\cdots x_{i_{t_\ell+1}}$. 
Note one can now reapply Theorem \ref{lowerboundresult} to this smaller
ideal.  At some point, the hypotheses of Theorem \ref{aw} will hold, which
stops our recursive calculation. 

This idea can be formally expressed as a formula, provided one is willing
to introduce even further notation (involving further subscripts on our subscripts).  Instead, we provide the following example in the hope of 
being more illuminating.

\begin{example}
Consider the monomal 
$$m= x_3x_4x_5x_8x_9x_{10}x_{48}x_{49}x_{50}x_{98}x_{99}x_{100} \in \mathbb{K}[x_1,\ldots,x_{100}]$$
and let $I = {\rm sfBorel}(m)$.
For this monomial $T(m) = (12,9,6,3)$ and $IT(m) = (100,50,10,5)$.
Since $i_{t_2}=i_6 = 10 < t_0 =12 < i_{t_1}=50$, then $\nu=11$ and Theorem \ref{lowerboundresult}
gives
$$\widehat{\alpha}(I) \geq \widehat{\alpha}(I_1) + 1+\frac{12-11}{98-10+1}=\widehat{\alpha}(I_1)+\frac{90}{89}$$
where $I_1 = {\rm sfBorel}(x_3x_4x_5x_8x_9x_{10}) = {\rm sfBorel}(m_1)$.
For this new monomial, we have $T(m_1) = (6,3)$ and $IT(m_1) = (10,5)$.
Again using Theorem \ref{lowerboundresult}, we get
$$\widehat{\alpha}(I_1) \geq \widehat{\alpha}(I_2) + 1+\frac{6-6}{10-6+1}
= \widehat{\alpha}(I_2) + 1$$
where $I_2 = {\rm sfBorel}(x_3x_4x_5)$.  Then by Theorem \ref{aw} (or in
this case, Corollary \ref{rationalcor}), we have $\widehat{\alpha}(I_2)
= \frac{5}{3}$.  Hence 
\[\widehat{\alpha}(I) \geq \frac{5}{3} + 1 + \frac{90}{89} = \frac{982}{267}.\]
Note that if we apply the upper bound of Theorem \ref{theo.aw}
we get
$$3.6904 \approx \frac{155}{42} \geq \widehat{\alpha}(I) \geq \frac{982}{267} \approx 3.6779.$$
\end{example}

We finish our paper with a result that
allows us to make small
changes to the generator of the 
square-free principal Borel without
changing the Waldschmidt constant.
\begin{theorem}
Let $I={\rm sfBorel}(x_{i_1}\cdots x_{i_{s-1}}x_{i_s})$ and 
$J={\rm sfBorel}(x_{i_1}\cdots x_{i_{s-1}}x_{{i_s}+r})$ for $r \in \mathbb{N}$. Then
     $\widehat\alpha(I)=\widehat\alpha(J).$
 \end{theorem}
 
 \begin{proof}
     Let $A$ be the matrix of associated primes of $I$ and let $\mathbf{y}=(y_1, \dots, y_{i_s})$ be an optimal solution to
     \begin{equation}\label{A1} \min\{\mathbf{1}^{T}\mathbf{x}\ |\ A\mathbf{x}\geq\mathbf{1}\}.
     \end{equation}
 First consider $r=1$ with $J={\rm sfBorel}(x_{i_1}\cdots x_{i_{s-1}}x_{{i_s}+1})$.  Let $A'$ be the matrix of associated primes of $J$.
   Any element of ${\rm ass}(J)$ not in ${\rm ass}(I)$  includes both $x_{{i_s}}$ and $x_{{i_s}+1}$ as generators, so the columns of $A'$ corresponding to $x_{{i_s}}$ and $x_{{i_s}+1}$ are identical. Let $\mathbf{y'} = (y_1',\dots,y'_{i_{s}+1})$ be an optimal solution to
     \begin{equation}\label{A2} \min\{\mathbf{1}^{T}\mathbf{x}\ |\ A'\mathbf{x}\geq\mathbf{1}\}.
     \end{equation}
and suppose for contradiction that $y_1' + \dots + y'_{i_{s}+1} = \widehat\alpha(J) < \widehat\alpha(I).$ But this means $(y_1',\dots,y'_{i_{s}}+y'_{i_{s}+1})$ is a feasible solution to (\ref{A1}), contradicting  $\mathbf{y}$ being optimal and showing $\widehat\alpha(J) \geq \widehat\alpha(I)$. Observing that $(y_1,\dots,y_{i_{s}},0)$ is a feasible solution to (\ref{A2}) gives $\widehat\alpha(I) \geq \widehat\alpha(J)$. This shows  $\widehat\alpha(I)=\widehat\alpha(J)$, and an inductive argument gives the result for $r \in \mathbb{N}$.
   \end{proof}


\end{document}